\documentclass[11pt,twoside,reqno]{amsart}

\usepackage{amssymb}
\usepackage{xcolor}
\usepackage[final]{hyperref}

\hypersetup{unicode= false, colorlinks=true, linkcolor=blue,
	anchorcolor=blue, citecolor=green, filecolor=red, menucolor=blue,
	pagecolor=blue, urlcolor=blue}

\def\const{\text{\rm const}}

\def\ti{\tilde}

\def\PW{\text{\rm PW}}

\def\to{\rightarrow}

\def\ms{\bigskip}
\def\ss{\smallskip}
\def\ms{\medskip}
\def\no{\noindent}

\def\R{{\mathbb R}}

\def\Z{{\mathbb{Z}}}
\def\C{{\mathbb{C}}}

\def\BB{{\mathcal B}}

\def\FF{{\mathcal F}}

\def\HH{{\mathcal H}}
\def\KK{{\mathcal K}}

\def\d{\delta}

\def\l{\lambda}
\def\g{\gamma}

\def\const{\text{\rm const}}

\theoremstyle{plain}
\newtheorem{lemma}{Lemma}
\newtheorem{theorem}{Theorem}
\newtheorem{corollary}{Corollary}
\newtheorem{proposition}{Proposition}

\newtheorem{remark}{Remark}

\numberwithin{equation}{section}

\title{Etudes in  the inverse spectral problem, II}

\begin{document}

	\author{N.~Makarov}
	\address{California Institute of Technology\\
		Department of Mathematics\\
		Pasadena, CA 91125, USA}
	\email{makarov@its.caltech.edu}
	
	\author{A.~Poltoratski}
	\address{University of Wisconsin\\ Department of Mathematics\\ Van Vleck Hall\\
		480 Lincoln Drive\\
		Madison, WI  53706\\ USA }
	\email{poltoratski@wisc.edu}
	\thanks{The second author was partially supported by
		NSF Grant DMS-2244801.}
	
	\begin{abstract} We apply the approach developed in our previous papers \cite{MIF1, MP3, etudes} to obtain examples of solutions to the inverse spectral problem (ISP) for the canonical Hamiltonian system. One of our goals is to illustrate connections of ISP with classical tools of analysis, such as the Hilbert transform and solutions to the Riemann-Hilbert problem. A key role in our study is played by the systems with homogeneous and quasi-homogeneous spectral measures. We show how some of such systems give rise to families of Bessel functions.
	\end{abstract}

\maketitle

\section{Introduction}

This note focuses on spectral problems for canonical Hamiltonian systems on the half-line. Our goal is to illustrate some of the methods and formulas developed in our earlier work with several examples and to outline connections with other problems of analysis. 

A {\it regular} half-line canonical (Hamiltonian) system is the equation
\begin{equation}
	\Omega\dot X=z\HH X\qquad {\rm on}\quad [0,\infty).\label{eq001A}
\end{equation}
Here the Hamiltonian $\HH=\HH(t)$ is  a  given $2\times 2$ matrix-function satisfying
$$\HH\in L^1_{\rm loc}[0,\infty),\quad \HH\ne 0\quad {\rm a.e.},\quad \HH\ge 0\quad {\rm a.e.}$$
The first relation means that the entries of $\HH$ are integrable on each finite interval. Systems satisfying this condition are called {\it regular}. 
The matrix $\Omega$ in \eqref{eq001} is the symplectic matrix $$\Omega=\begin{pmatrix} 0&1\\-1&0\end{pmatrix},$$ and $z\in\C$ is the 'spectral parameter'. The unknown function 
$X=X(t,z)$ is a two-dimensional vector-function on $[0,\infty)$.

The Krein-de Branges theory translates spectral problems for canonical systems in the language of complex analysis. It can be applied  to a wide range of problems in mathematical physics, including Schr"odinger operators, Dirac systems, and string equations. For a recent account of the theory see \cite{Rem}.

In the next section we survey the basic notions of the theory which will be used in this note. One of our goals is to illustrate the formulas and 
algorithms obtained in our earlier work \cite{MIF1, etudes, MP3} with several examples and connections. Special attention is paid to systems with homogeneous spectral measures.

The first part of our paper focuses on an application of the Riemann-Hilbert method to  an inverse spectral problem. We show how the 'truncated Toeplitz equation' obtained in \cite{etudes} can be transformed into a Riemann-Hilbert problem on a finite interval $[-t,t]$ for the simplest non-even example of a homogeneous spectral measure, 
 $\mu = c_1 m + c_2 \sigma$ where $\sigma(x) = \text{sign}(x)$, $m$ is the Lebesgue measure on $\R$ and $c_1,c_2$ are real constants, 
 $c_1>|c_2|$. 
 

The second part investigates homogeneous and quasi-homogeneous spectral measures in more generality. We establish that a measure $\mu$ is homogeneous if and only if all its associated de Branges spaces are homogeneous and we use the  scaling relation  for reproducing kernels to solve an inverse spectral problem. This part of the paper utilizes the generalized Hilbert transform defined in \cite{etudes}.
For quasi-homogeneous measures of order $\nu$, we obtain the universal relation for the corresponding reproducing kernels.

The third section connects homogeneous canonical systems with Bessel functions. We show that systems with Hamiltonians $H(t) = \text{diag}(t^m, t^{-m})$ lead to solutions expressible in terms of Bessel functions. Specifically, we prove that for such systems, the transfer matrix entries satisfy $A = F_{\nu-1}(zt)$ and $C = t^{2\nu} z F_\nu(zt)$ where $F_\nu$ is related to the Bessel function $J_\nu$ through $J_\nu(\lambda) = \lambda^\nu F_\nu(\lambda)$.

Throughout our analysis, we emphasize computational aspects, providing explicit formulas for Hamiltonian recovery.

The paper is organized as follows: Section 2 surveys the basics of the Krein-de Branges theory and the necesary formulas from our previous work. Section 3 presents the Riemann-Hilbert formulation of a particular  inverse spectral problem. Section 4 focuses on a further study of homogeneous and quasi-homogeneous measures. Section 5 connects canonical systems with Bessel functions and provides explicit solutions. 


\section{Preliminaries}

\subsection{Canonical systems and de Branges spaces} Instead of a two-dimensional vector function $X$ one may look for a $2\times 2$ matrix-valued solution
$M=M(t,z)$ solving \eqref{eq001A}. Such a matrix valued function satisfying the initial condition  
$M(0,z)=I$ is called the \textit{transfer matrix} or the \textit{matrizant} of the system. The columns of the transfer matrix $M$ are the solutions for the system \eqref{eq001A} satisfying the initial conditions $\begin{pmatrix}1\\0\end{pmatrix}$ (Neumann) and $\begin{pmatrix}0\\1\end{pmatrix}$ (Dirichlet) at 0.
As a general rule we denote
\begin{equation}M=\begin{pmatrix} A& B\\C&D\end{pmatrix}.\label{eqTM}\end{equation}

An entire function $F(z)$ belongs to the Hermite-Biehler (HB) class if 
$$|F(z)|>|F(\bar z)|\textrm{ for all }z\in \C_+.$$
We say that an entire function is real if it is real on $\R$.

For each fixed $t$, the entries of the transfer matrix $M$ of the system \eqref{eq001}, $A(z)=A(t,z)\equiv A_t(z)$, $B(z), C(z)$ and $D(z)$ are real entire functions. The functions
$$E:= A-iC,\qquad \tilde E:= B-iD$$
belong to the Hermite-Biehler class; see for instance \cite{Rem}.

For an entire function $G$ we denote by $G^\#$ its Schwarz reflection with respect to $\R$,
$G^\#(z)=\bar G(\bar z)$. We denote by $H^2$ the standard Hardy space in the upper half-plane.

For every Hermite-Biehler function $F$ one can consider the de Branges  (dB) space $\BB(F)$, a Hilbert space of
entire functions defined as
$$\BB(F)=\left\{G\ | G\text{ is entire, }\frac GF,\ \frac {G^\#}F\in H^2\right\}.$$
The Hilbert space structure in $\BB(F)$ is inherited from $H^2$:
$$<G,H>_{\BB(F)}=\left<\frac GF,\frac HF\right>_{H^2}=\int_{-\infty}^{\infty}G(t)\bar H(t)\frac{dt}{|F(t)|^2}.$$

The space $\BB(E)$ is a reproducing kernel Hilbert space, i.e., for each $\l\in\C$ there exists  $K(\l,\cdot)\in \BB(E)$ such that for any $F\in \BB(E)$,
$$F(\l)=<F,K(\l,\cdot)>_{\BB(E)}.$$
The function $K(\l,z)$ is called the reproducing kernel  (reprokernel) for the point $\l$. In the case of the dB-space $\BB(E)$, $K(\l,z)$  has the formula
$$K(\l,z)=\frac{1}{2\pi i}\frac{  E(z) E^\#(\bar \l)- E^\#(z)E(\bar \l)}{\bar \l -z }=\frac{1}{\pi }\frac{A(z) C(\bar \l) - C(z)A(\bar \l))}{\bar \l-z},
$$
where $A=(E+E^\#)/2$ and $C=(E^\#-E)/2i$ are real entire functions such that $E=A-iC$. 

The functions $E,\ \ti E$ corresponding to a canonical system \eqref{eq001A} give rise to the family
of dB-spaces 
$$\BB_t=\BB(E(t,\cdot)),\qquad \tilde\BB_t=\BB(\tilde E(t,\cdot)).$$

A value $t$ is $\HH$-regular if it does not belong to an open interval on
which $\HH$ is a constant matrix of rank one. The spaces $\BB_t,\ti \BB_t$
form {\it chains}, i.e., $\BB_s\subsetneq \BB_t$ for $s<t$ and the inclusion
is isometric for regular $t$ and $s$.

Special role in our formulas is played by the kernels at 0:  $K_t(z,0)$ and $\tilde K_t(z,0)$
We denote $k_t(z)=K_t(z,0)$.

%
%

\ms \subsection{Spectral measures}\label{secSM} There are several ways to introduce spectral measures of canonical systems. We'll make a simplifying assumption that the system has no "jump intervals", i.e., intervals on which the Hamiltonian is rank one and constant. In this case all $t\in [0,\infty)$ are
$\HH$-regular and all inclusions $\BB_s\subset\BB_t, \ti\BB_s\subset\ti\BB_t$
are isometric.
We can make this assumption because we will be mostly concerned with the case $\det\HH\neq 0$ a.e.


A  measure $\mu$ on $\R$ is called Poisson-finite ($\Pi$-finite) if  $$\int\frac{d|\mu|(x)}{1+x^2}<\infty.$$
A measure on $\hat R=\R\cup \{\infty\}$ is $\Pi$-finite if it is the sum of a $\Pi$-finite measure on $\R$ and a finite point mass at infinity.

By definition, a positive measure $\mu$ on $\R$ is a spectral measure of the CS \eqref{eq001} with the initial condition $\begin{pmatrix}1\\0\end{pmatrix}$ at $t=0$ if
$$\forall t,\qquad \BB_t\stackrel{\rm iso}{\subset} L^2(\mu).$$
(The definition is slightly more complicated in presence of jump intervals for the Hamiltonian, which we do not allow in this paper.)
It is well-known that spectral measures of regular CS are $\Pi$-finite; see for instance \cite{Rem}. In a similar way, using $\ti B_t$, one can define a spectral measure
$\ti \mu$ for the initial condition  $\begin{pmatrix}0\\1\end{pmatrix}$. 

Conversely, one of the main results of the Krein-de Branges theory says that every positive $\Pi$-finite
measure is a spectral measure of a regular CS. In general, the corresponding system may not satisfy $\det \HH\neq 0$ a.e., the restriction we are assuming in this article. Also,
HB functions corresponding to the systems considered in this paper have no zeros on the
real line. We will assume this restriction in our general
discussions of dB-spaces. (If $E$ vanishes at some point of $\R$, then all functions in $\BB(E)$ must vanish at the same point, as follows from the definition. PW-type spaces discussed in this note clearly do not have such a property.)

Every regular canonical system has a spectral measure; in fact for $\mu$ we can take any limit point of the family of measures $|E_t|^{-2}$ as $t\to \infty$, \cite{dB}. The spectral measure may or may not be unique. It is unique iff
\begin{equation*}\int{\rm trace}\ \HH(t)dt=\infty.\end{equation*}
The case when the spectral measure is unique is called the limit point case and the case when
it is not, the limit circle case.

Finally, let us mention that spectral measures are invariant with respect to  'time' parametrizations, i.e., a  change of variable $t$ in the initial system \eqref{eq001} via an increasing homeomorphism $t\mapsto s(t)$ does not
change the spectral measure.

\ms\subsection{Paley-Wiener spaces}\label{secPW}

We will use the following
definition for the Fourier transform in $L^2(\R)$:
$$ (\FF f)(\xi)\equiv \hat f(\xi)=\frac1{\sqrt{2\pi}}\int e^{-i\xi x}f(x)dx,$$
first defined on test functions and then extended to a unitary operator
$\FF L^2(\R)\to L^2(\R)$ via Parseval's theorem. The Paley-Wiener space $\PW_t$ of entire functions is defined as the image
$${\rm PW}_t=\FF L^2[-t,t].$$
By the Paley-Wiener theorem, $\PW_t$ can be equivalently defined as
the space of entire functions of exponential type at most $t$ which belong
to $L^2(\R)$. The Hilbert structure in $\PW_t$ is inherited from $L^2(\R)$. 

Paley-Wiener spaces appear as the de Branges chain $\BB_t$ for the free system $\HH=I$.

\ms\subsection{PW systems and PW measures}

Let $\HH$ be a Hamiltonian of a canonical system \eqref{eq001A}. We say that $\HH$ is of PW-type ($\HH\in \PW$)  if for any $t>0$ there exists $s=s(t)>0$ such that $s(t)\to\infty$ as $t\to\infty$ and the spaces $\BB_t(\HH)$ and ${\rm PW}_s$ are equal as sets (with possibly different norms), 
\begin{equation}\BB_t(\HH)\doteq {\rm PW}_s.\label{eq20}\end{equation}
We call the corresponding system a PW-type system.

All regular Dirac systems (systems with locally summable potentials) are of PW-type as follows from Lemma 3.1 in \cite{etudes}
and the results of \cite{LS}; see also \cite{BR}. The class of all $PW$-type systems is significantly broader than the class considered in the classical Gelfand-Levitan theory; see \cite{etudes} for further discussion.

By definition, a positive measure $\mu$ on $\R$ is  PW-sampling  ($\mu\in\PW$) if it is sampling for all Paley-Wiener spaces PW$_t$:
$$\forall t\quad\exists C>0,\qquad \forall f\in {\rm PW}_t, \qquad C^{-1}\|f\|\le \|f\|_{L^2(\mu)}\le C\|f\|.$$

The set of all PW-sampling measures admits the following elementary description:

Given $\mu$ and $\delta>0$ we say that an interval $l\subset\R$ is $\delta$-{\it massive} with respect to $\mu$ if
$$\mu(l)\ge \delta\textrm{ and } |l|\ge\delta.$$
The $\delta$-{\it capacity}   of an interval $I\subset\R$ with respect to $\mu$, denoted by $C_\delta(I)$, is the maximal number of disjoint $\delta$-massive intervals intersecting $I$.

\ms

\begin{theorem}[\cite{etudes}]\label{PWmu} $\mu$ is  PW-sampling if and only if
	
	(i) For any $x\in \R$, $\mu(x,x+1)\le \const$;
	
	(ii) For any $t>0$ there exist $L$ and $\d$ such that for all $I,\ |I|\ge L$,\newline
	$$C_\delta(I)\ge t|I|.$$
\end{theorem}

As the most basic example, any measure of the form $\rho(x)m(x)$, where $0<c<\rho(x)<C<\infty$ is PW-sampling. For further examples see \cite{etudes}.

PW-systems and PW-sampling measures are related via the following statement. For $\mu\in$PW we denote by ${\rm PW}_t(\mu)$ the
Paley-Wiener spaces  $PW_t $ endowed with the equivalent norm from $L^2(\mu)$.

\begin{proposition}
[\cite{etudes}]Suppose $\det(\HH)=1$ a.e., and let $\mu$ be the (unique) spectral measure of the corresponding CS. Then
	$$\mu\in{\rm(PW)}\qquad \Leftrightarrow\qquad \HH\in{\rm(PW)}.$$
	Moreover, if either holds, then
	$$\forall t,\quad \BB_t(\HH)\doteq {\rm PW}_t(\mu).$$\end{proposition}

\ms\subsection{ Det-normalization} We have the following \cite{etudes}:

\begin{theorem}\label{t0001} If $\HH$ is of PW-type, then
	\begin{equation}
		\det \HH\ne 0\quad{\rm a.e.},\qquad \int_0^\infty\sqrt{\det\HH(t)}dt=\infty. \label{eq005A}
	\end{equation}
\end{theorem}

	The change of time $t\mapsto s=s(t)$ in the above theorem allows us to transform any PW-type system, or more generally any canonical system satisfying 
\eqref{eq005A}
	 into a canonical system with $$\det \HH=1\quad{\rm a.e.}$$
	We will call such systems {\it det-normalized}. A regular (locally summable) Hamiltonian will remain regular under det-normalization.
	
	Throughout the rest of the paper, all systems considered are assumed to be det-normalized.

	\ms\subsection{Inverse problem: recovery of $h_{11}$ ($h^\mu$)}
	For PW-systems the leading term of the Hamiltonian can be recovered from the reproducing kernels $k_t\in {\rm PW}_t(\mu)$:

	\begin{theorem}[\cite{etudes}]\label{t3} Let $\mu\in\PW$ be the spectral measure of a system \eqref{eq001A} with the Hamiltonian $\HH$. Then $t\mapsto k_t(0)$ is an absolutely continuous
		function and 
		\begin{equation}h_{11}(t)=h^\mu(t):=\pi\frac d{dt} k_t(0).\label{eq003A}\end{equation}
	\end{theorem}
	
	Systems with even spectral measures have diagonal Hamiltonians. If, in addition, the system is det-normalized, then 
	$h_{22}=1/h_{11}$ and the ISP is solved with the recovery of $h_{11}$. For non-diagonal cases further analysis must be conducted, see below.

\ms\subsection{Generalized Hilbert transform}\label{secHT} For a $\Pi$-finite measure $\mu$ and $f\in L^2(|\mu|)$ we will use the notation
$$ K(f\mu)(z)=\frac1{\pi }\int\frac{f(s)~d\mu(s)}{s-z},$$
and
$$\KK\mu(z)=\frac1{\pi }\int\left[\frac1{s-z}-\frac s{1+s^2}\right]~d\mu(s),$$
where $z\in\C\setminus\R$. If $f\in L^2(\mu)$ is an entire function, then we define
$$H^\mu f=K(f\mu)-f\KK\mu.$$
It is clear that $H^\mu f$ extends to an entire function:
$$
(H^\mu f)(z)=\frac1\pi\int\left[\frac{f(s)-f(z)}{s-z}+\frac{s f(z)}{1+s^2}\right]~d\mu(s).$$

As was shown in \cite{etudes}, $H^\mu$ plays an important role in the recovery of the off-diagonal elements of the Hamiltonian:

\ms	\subsection{Inverse problem: recovery of $h_{12}=h_{21}$ ($g^\mu$)}
	
	For the off-diagonal terms we use

\begin{theorem}[\cite{etudes}] \label{t007} Let $\mu\in\PW$. Consider the reproducing kernels at 0, $k_t\in\PW_t(\mu)$. Define $\tilde l_t=H^\mu k_t$. Then $\mu$ is the spectral measure of the Hamiltonian
	$$\HH=\begin{pmatrix}h^\mu&g^\mu\\g^\mu&\frac{1+g^2_{\mu}}{h^\mu}\end{pmatrix},$$
	where
	$$g^\mu(t):=\pi\frac d{dt} \tilde l_t(0).$$
\end{theorem}

\ms\subsection{Equations for the Fourier transform of $k_t$}\label{secFt} How to compute the functions $h^\mu$ and $g^\mu$
from Theorem \ref{t007}? Sometimes  it is helpful to work with the functions
$$\psi_t:=\hat k_t,$$
so that
$$\frac1{\sqrt{2\pi}}\int_\R\psi_t=k_t(0).$$


If $f\in {\rm PW}_t$, then
$$f(0)=\frac1{\sqrt{2\pi}}\int_{-t}^t\hat f(\xi)~d\xi$$
while  also
$$ f(0)=(f, \mathring k_t)_{PW_t}=\left(\hat f,\FF(\mathring k_t)\right)_{L^2(-t,t)}.$$
This leads us to the well known formula for the Fourier transform of the sinc function:
$$\FF \mathring k_t =\FF\left(\frac{\sin tz}{\pi z}\right)
=  \frac1{\sqrt{2\pi}}~1_{(-t,t)}.$$

The Fourier transform $\psi_t$ of the reprokernel $k_t$ can be found using the following statement.

\begin{theorem}[\cite{etudes}] \label{tTT} $\psi=\psi_t$ satisfies
	\begin{equation}\psi\ast \hat\mu=1 \quad  {\rm on}\quad (-t,t)\label{eq002A}\end{equation}
	and
	$$\psi=0 \quad  {\rm on}\quad \R\setminus [-t,t].$$\end{theorem}

(In general, all Fourier transforms and convolutions throughout the rest of the paper are understood in the sense of distributions.)

\section{Riemann-Hilbert}\label{RH}


\ms \subsection{Hilbert transform in ISP} 
We start with an example of an ISP for one of the simplest non-even spectral measures. 

First,
denote
$$\sigma (x)={\rm sign}(x),$$
where the Fourier transform is understood in the sense of distributions.
Then 
$$\hat\sigma(t)=\sqrt{\frac 2\pi}~\frac 1{it}.$$

We will use the notation $Hf$ for the standard Hilbert transform of a function $f$ on $\R$:
$$Hf(x)=\frac 1\pi~\text{p.v.}~\int_{\R}\frac {f(t)}{t-x}dt.$$
	In terms of convolutions,
	$$Hf=\frac 1\pi f\ast \frac 1t.$$
	
$$H\psi_t=\frac1\pi~\psi_t\ast\frac1t.$$

 Let now define a measure on $\R$ as
\begin{equation}\mu=c_1m+c_2\sigma.\label{eq005}\end{equation}
with $c_1>|c_2|$, $c_1,c_2\in \mathbb R$. Clearly, $\mu\in$(PW).  Also, from the well-known rescaling properties of the Hamiltonians
and spectral measure, it follows  that $h_{11}(t)$ is a constant function. First, in our calculations below we aim to find that constant.

Consider the  chain of (regular) de Branges speces $\BB_t$ corresponding to $\mu$ and the family of reproducing kernels  at zero, $k_t\in\BB_t$.
Let as before $\psi_t=\hat{k_t}$.
Since
$$\psi_t\ast\hat\mu=\sqrt{2\pi}\left[c_1\psi_t-ic_2H\psi_t\right],$$
from \eqref{eq002A} we obtain
$$\sqrt{2\pi}\left[c_1\psi_t-ic_2H\psi_t\right]=1\quad {\rm on}\quad(-t,t).$$

To solve this equation for $\psi_t$ we need to deviate into classical complex analysis

\ms\subsection{Plemelj theorem} We introduce the function
\begin{equation}C_t(z)=\frac1{2\pi i}\int_{-t}^t\frac{\psi_t(s)~ds}{s-z},
\quad z\in \hat{\mathbb C}\setminus [-t,t],\label{eq002}\end{equation}
and denote its boundary values by
$C^\pm_t(x)$, $-t\le x\le t$:
$$C^\pm_t(x)=\lim_{y\to\pm 0}C_t(x+iy).$$ 

\begin{theorem}
$$\psi_t=C^+_t-C^-_t,\qquad -i(H\psi_t)=C^+_t+C^-_t.$$
\end{theorem}

\ms\no It follows that \eqref{eq002A} has the form
\begin{equation}
	C^+_t=GC^-_t+g,\label{eq001}
\end{equation}
where $G$ and $g$ are the numbers
$$G=\frac{c_1-c_2}{c_1+c_2},\qquad g=\frac1{\sqrt{2\pi }}\frac1{c_1+c_2}.$$
   
\ms\subsection{THe Riemann-Hilbert problem} 
The above equation \eqref{eq002A} is a particular case of the classical Riemann-Hilbert problem.
Given two function $F(s)$ and $f(s)$ on $[-t,t]$, we want to find an analytic function $\Phi=\Phi_t(z)$ in $\hat{\mathbb C}\setminus [-t,t]$
such that $\Phi(\infty)=0$ and
$$\Phi^+=F\Phi^-+f\quad {\rm on}\quad [-t,t].$$

The following statement can be verified by direct calculations:

\begin{theorem} Denote 
\begin{equation}X_t(z)=\exp\left\{\frac1{2\pi i}\int_{-t}^t\frac{\log F(s)}{s-z}~ds\right\}.\label{eq003}\end{equation}
Then
\begin{equation}\Phi_t(z)=X_t(z)\left[\frac1{2\pi i}\int_{-t}^t\frac{f(s)~ds}{(s-z) X_t^+(s)}        \right].\label{eq004}\end{equation}
\end{theorem}

\ms\subsection{Computation of $\int_{-t}^t\psi_t$}  

We will use the theorem with $F(s)=G$ and $f(s)=g$ to solve \eqref{eq001}.
As $z\to\infty$, \eqref{eq002} implies
$$ C_t(z)~\sim -\frac1z ~\frac1{2\pi i}\int_{-t}^t\psi_t$$
and, from \eqref{eq004},
$$C_t(z)=\Phi_t(z)~\sim -\frac1z ~\frac1{2\pi i}\int_{-t}^t \frac{g}{X_t^+}(s).$$
 
By direct calculations,
$$\int_{-t}^t \frac{1}{s-z}ds=\log\left(\frac{z-t}{z+t}\right).$$
The boundary limits of the last function on $(-t,t)$ are
$$\left(\log\left(\frac{z-t}{z+t}\right)\right)_\pm=\pm\pi i + \log\left|\frac{z-t}{z+t}\right|.$$
From \eqref{eq003} we obtain
$$X^+_t(s)=\exp\left[D\left(\frac 12   +\frac1{2\pi i}\log\left|\frac{s-t}{s+t}\right|\right)\right],$$
where $D=\log G$.
Let us now compute
$$I=\int_{-t}^t \frac 1{X^+_t(s)}ds. $$

First, note that
\[
\frac{1}{X^+_t(s)} = e^{-D/2} \cdot \left( \frac{t-s}{t+s} \right)^{\frac{D i}{2\pi}}.
\]

To calculate the integral, let \( u = \frac{s}{t} \), so \( s = t u \), \( ds = t  du \), and when \( s = -t \), \( u = -1 \); when \( s = t \), \( u = 1 \). Then
\[
\frac{t-s}{t+s} = \frac{1 - u}{1 + u}
\]
and
\[
I = e^{-D/2} \int_{-1}^{1} \left( \frac{1 - u}{1 + u} \right)^{\frac{D i}{2\pi}} t  du = t e^{-D/2} \int_{-1}^{1} \left( \frac{1 - u}{1 + u} \right)^{\frac{D i}{2\pi}} du.
\]
Further,
let \( w = \frac{1 - u}{1 + u} \), then \( u = \frac{1 - w}{1 + w} \), \( du = -\frac{2}{(1 + w)^2} dw \).  
When \( u = -1 \), \( w = \infty \); when \( u = 1 \), \( w = 0 \).  
Thus,
\[
\int_{-1}^{1} \left( \frac{1 - u}{1 + u} \right)^{\frac{D i}{2\pi}} du = \int_{\infty}^{0} w^{\frac{D i}{2\pi}} \left( -\frac{2}{(1 + w)^2} dw \right) = \int_{0}^{\infty} w^{\frac{D i}{2\pi}} \frac{2}{(1 + w)^2} dw.
\]
Furthermore
\[
I = t e^{-D/2} \cdot 2 \int_{0}^{\infty} \frac{w^{\frac{D i}{2\pi}}}{(1 + w)^2} dw.
\]

Let \( k = \frac{D i}{2\pi} \). Then the last integral becomes
\[
J = \int_{0}^{\infty} \frac{w^{k}}{(1 + w)^2} dw.
\]
This is a Beta integral:
\[
J = B(k + 1, 1 - k) = \frac{\Gamma(k + 1) \Gamma(1 - k)}{\Gamma(2)} = \Gamma(k + 1) \Gamma(1 - k), \quad \text{since } \Gamma(2) = 1.
\]
Hence,
\[
J = \Gamma\left(1 + \frac{D i}{2\pi}\right) \Gamma\left(1 - \frac{D i}{2\pi}\right).
\]

Using the Gamma function identity,
\[
\Gamma(1 + z) \Gamma(1 - z) = \frac{\pi z}{\sin(\pi z)}
\]
with \( z = \frac{D i}{2\pi} \), we obtain
\[
\Gamma\left(1 + \frac{D i}{2\pi}\right) \Gamma\left(1 - \frac{D i}{2\pi}\right) = \frac{\pi \cdot \frac{D i}{2\pi}}{\sin\left( \pi \cdot \frac{D i}{2\pi} \right)} = \frac{\frac{D i}{2}}{\sin\left( \frac{D i}{2} \right)}.
\]
But \( \sin(i x) = i \sinh(x) \) and
\[
\sin\left( \frac{D i}{2} \right) = i \sinh\left( \frac{D}{2} \right).
\]
Therefore,
\[
J = \frac{\frac{D i}{2}}{i \sinh\left( \frac{D}{2} \right)} = \frac{D}{2 \sinh\left( \frac{D}{2} \right)}.
\]

Combining the calculations we get
\[
I = t e^{-D/2} \cdot 2 \cdot J = t e^{-D/2} \cdot 2 \cdot \frac{D}{2 \sinh\left( \frac{D}{2} \right)} = t e^{-D/2} \cdot \frac{D}{\sinh\left( \frac{D}{2} \right)}.
\]
Since
\[
\frac{1}{\sinh\left( \frac{D}{2} \right)} = \frac{2}{e^{D/2} - e^{-D/2}},
\]
we obtain
\[
I =\int_{-t}^{t} \frac{1}{X^+_t(s)}  ds = t e^{-D/2} \cdot D \cdot \frac{2}{e^{D/2} - e^{-D/2}} = \frac{2 t D e^{-D/2}}{e^{D/2} - e^{-D/2}} = \frac{2 t D}{e^{D} - 1}.
\]

To finish the calculation of $h_{11}$,
 
 $$h_{11}(t)=\pi \frac{\partial}{\partial t} \int_{-t}^t \psi_t=\pi \frac{\partial}{\partial t}\int_{-t}^{t} \frac{g}{X^+_t(s)} =\frac {2\pi gD}{e^D-1}=\frac{2\pi g\log G}{G-1}=\sqrt{\frac\pi 2}\frac 1{c_2}\log\left(\frac{c_1+c_2}{c_1-c_2}\right).$$
 

We will return to this example and calculate the full Hamiltonian (up to a constant) at the end of the next section.

\section{Homogeneous measures and spaces} \label{Hom}
\ms \subsection{Homogeneous spectral measure}  
For a measure $\mu(x)$ on $\R$ and $t>0$ we denote by $\mu_t(x)=\frac 1t\mu(tx)$ the measure such that for any Borel $B\subset \R$,
$$\mu_t(B)=\frac 1t\mu(tB).$$

A measure $\mu$ is homogeneous if 
$$\forall t>0,\qquad \mu_t(x)=\mu(x).$$
It is easy to show that a homogeneous measure must be absolutely continuous, $d\mu(x)=\rho(x)dx$,  where
$\rho(x)$  is constant on  $\R_+$ and on $\R_-$. We will assume that  the constants are strictly positive, so that   $\mu\in(\PW)$.
Note that all examples of such measures measures are given 
by \eqref{eq005}.

As usual we denote by $k_t(z)$ the reprokernels of the dB spaces $\PW_t(\mu)$ at zero.
\begin{theorem}   If $\mu\in (\PW)$ is homogeneous, then we have the identity
$$k_t(z)=tk_1(tz).$$
 \end{theorem}
 
 \ss\no{\bf Example.}  If $d\mu=dx$, then 
 $$k_t(z)\equiv k_t^0(z)=\frac1\pi~\frac {\sin tz}z,$$
 and 
 $$k_t^0(z)=tk_1^0(tz).$$ 
 \ms
 \begin{proof} Let $F\in\BB_1=\PW_1$ ("as sets"). Then $$G:=\sqrt t F(tz)\in \PW_t=\BB_t.$$
 Since
 $$F(0)=\frac1{\sqrt t}G(0),$$
 we have
 $$\int F(y)k_1(y)\rho(y)dy=(F,k_1)_{\BB_1}=\frac1{\sqrt t} (G,k_t)_{\BB_t}=$$
 $$\int F(tx)k_t(x)\rho(x)dx=\int F(y)k_t(y/t)\rho(y)\frac{dy}t$$
Denote
$$K(y)=k_1(y)-\frac1tk_t(y/t).$$
Then $ K\in \PW_1$ and 
$$\forall F\in \PW_1,\quad \int FK~d\mu=0.$$
Setting $F=K$ we get $K=0$.
 \end{proof}
 
\ms\no Following \cite{dB},  we say that a dB space $\BB=\BB(E)$ is homogeneous   if for all $t\in(0,1)$
$$F\in \BB\quad\Rightarrow\quad  t ^{1/2} F(tz)\in\BB$$
and both functions have the same norm in $\BB$.

\begin{theorem}\label{thm2}  $\mu\in(PW)$ is homogenous iff all its dB spaces are homogenous. \end{theorem}
\begin{proof} 
	
Suppose first that $\mu$ is homogeneous.	Let $F\in \BB_a$ and $0<t<1$. Then
$F\in \PW_a$ and 
$$G:=\sqrt t F(tz)\in\PW_{ta}\subset \PW_a=\BB_a$$
(last equation means "equal as sets"). It remains to show
$$\|G\|_{\BB_a}=
\|F\|_{\BB_a},$$  or equivalently
$$\int |G|^2~d\mu=\int |F|^2~d\mu.$$
We have
$$\int |G|^2~d\mu=\int tF(tx)^2~\rho(x)dx=\int F(y)^2~\rho(y)dy,$$
which means that all $\BB_a$ are homogeneous.

Suppose now that all $\BB_a$ are homogeneous. Let us show that $\mu(x)=\mu(tx)$ for all positive $t$.
If $F\in PW_a$ then
$$\int |F(x)|^2d\mu(x)=\int t|F(tx)|^2d\mu(x)=\int |F(y)|^2d\mu_t(y).$$
Hence, the measures $\mu(x)$ and $\mu_t(x)$ define the same norms on every $PW_a$ for any $c>0$.
Now, let us consider a det-normalized canonical system with the spectral measure $\mu$. 
By the definition of the spectral measure, $\mu_t(x)$ is also a spectral measure for the same system for any $t>0$.
From the uniqueness of the spectral measure, $\mu(x)=\mu_t(x)$.

\end{proof}

\begin{corollary} Let $d\mu=\rho(x)dx$.
The following  three conditions are equivalent:

\ms\no (i) $\rho(x)=\rho(tx)$ for all $t\in (0,1)$;

\ms\no (ii) $\BB_t$ is homogeneous  for all $t>0$;

\ms\no(iii) $k_t(z)=tk_1(tz)$  for all $t\in (0,1)$.

\end{corollary}

\ms\subsection{Quasi-homogeneuous spectral measures} By definition, a measure $\mu(x)$ is quasi-homogeneous of order $\nu$ if
for all $t>0$,
$$t^{1+2\nu}\mu(x)=\mu_t(x).$$
Once again, it is not difficult to prove that such measures are absolutely continuous. Their densities $\rho$ must satisfy
$$\forall t>0,\quad t^{1+2\nu}\rho(x/t)=\rho(x),$$
or equivalently 
$$\forall t>0,\quad t^{1+2\nu}\rho(y)=\rho(ty).$$

One can show that  quasi-homogeneous  measures form a two parameter family:
$$\rho(x)=\begin{cases}x^{1+2\nu}\rho(1),\quad x>0\\
| x|^{1+2\nu}\rho(-1),\quad x<0
\end{cases},$$
with arbitrary positive constants $\rho(\pm 1)$. As usual, a special case among spectral measures of canonical systems is occupied by even measures, which correspond to the case $\rho(1)=\rho(-1)$. 

The measures are not Paley-Wiener unless $\nu=-1/2$.  Quasi-homogeneous measures are locally finite (on $\R$) iff $\nu>-1$, and Poisson-finite,
$$
\int^\infty\frac{d\mu(x)}{1+x^2}<\infty,$$
iff $\nu<0$. For this reason, we'll be considering only the case 
$$-1<\nu<0.$$
The value $\nu=-1/2$ was discussed in the previous subsection.

Similarly to Theorem \ref{thm2}  one can prove:

\begin{theorem}\label{t10}
The following  conditions are equivalent:

\ms\no (i) $\rho(tx)=t^{1+2\nu}\rho(x)$ for all $t\in (0,1)$;

\ms\no (ii) all $\BB_t$ are homogeneous of order $\nu$ for all $t>0$;

\ms\no(iii) $k_t(z)=t^{2+2\nu}k_1(tz)$  for all $t\in (0,1)$

\end{theorem}

\ms\subsection{Solution of ISP for a homogeneous system in PW-case} We will use 
$$k_t(z)=tk_1(tz),$$
in particular,
$$k_t(0)=C_1t,\qquad C_1:=k_1(0),$$
to compute the functions $h^\mu(t)$ and  $g^\mu(t)$ (up to a constant). Obviously,
$$h^\mu(t)=\pi\dot k_t(0)=\pi C_1.$$
For the second function we apply the generalized Hilbert transform:
$$g^\mu(t)=\pi \dot l_t(0),\qquad l_t(0)=(T^\mu k_t)(0).$$ 
Recall
$$(T^\mu k_t)(0)=\frac1\pi\int\left[ \frac{k_t(x)-k_t(0)}x+\frac{xk_t(0)}{1+x^2} \right]   ~d\mu(x)$$
$$=\frac1\pi\int\left[ \frac{tk_1(tx)-tk_1(0)}x+\frac{txk_1(0)}{1+x^2} \right]   ~\rho(x)~dx$$
$$=\frac t\pi\int\left[ \frac{k_1(y)-k_1(0)}y+\frac{yk_1(0)}{y^2+t^2} \right]   ~\rho(y)~dy$$$$= At+\frac{ tC_1}\pi B,$$
where
$$A=\frac 1\pi\int\left[ \frac{k_1(y)-k_1(0)}y+\frac{yk_1(0)}{1+y^2} \right] ~\rho(y)dy$$
and
$$B=\int y\left[\frac1{y^2+t^2}-\frac1{1+y^2}\right]~\rho(y)dy.$$

\ms\no Thus $A$ and $C_1$ are constants but $B=B(t)$. In fact,
$$B=-\left[\rho(1)-\rho(-1)\right]\log t$$
and
$$g^\mu(t)=C-\frac 1\pi C_1\left[\rho(1)-\rho(-1)\right]\log t$$
for some constant $C$. 

For det-normalized systems, $$h_{22}=\frac{1-h_{12}^2}{h_{11}}=\frac{1-(C-C_2\log t)^2}{C_1},$$
where $C_2=\frac 1\pi C_1\left[\rho(1)-\rho(-1)\right]$.

This results in the Hamiltonian
$$\HH=\begin{pmatrix} C_1& C-C_2\log t\\C-C_2\log t&\frac{1-(C-C_2\log t)^2}{C_1}\end{pmatrix}.$$

Returning to the example of Section \ref{RH}, with the spectral measure $\mu$ defined by \eqref{eq005}, we obtain
the solution to the ISP, with a free  constant parameter $C$, by putting
$$C_1= \sqrt{\frac\pi 2}\frac 1{c_2}\log\left(\frac{c_1+c_2}{c_1-c_2}\right), C_2= \frac1{\sqrt{2\pi}}\log\left(\frac{c_1+c_2}{c_1-c_2}\right)$$
into the previous formula.

As was shown in \cite{etudes}, every spectral measure $\mu$ gives rise to a one-parameter family of Hamiltonians $\HH^\mu$. 
The above formula gives a general solution to the ISP  for $\mu$ defined by \eqref{eq005}.

\ms\section{Bessel functions and canonical systems} 

In this section we give an example of a direct spectral problem which involves families of Bessel functions.

\ms\subsection{Bessel's family}

\begin{lemma} Suppose $F(t)$ satisfies 
\begin{equation}t^2\ddot F(t)+t\dot F(t)+(t^2-\nu^2)F(t)=0\label{eq007}\end{equation}
for $t>0$. Let $\kappa>0, \beta>0$, $\alpha$ be given real numbers. Then the function
$$y(t)=t^\alpha F(\kappa t^\beta) $$
solves the equation
\begin{equation}\qquad t^2 \ddot y+at\dot y+(b+c^2t^{2\beta})y=0\label{eq006}\end{equation}
with
\begin{equation}\qquad a=1-2\alpha,\qquad  b=\alpha^2-\beta^2\nu^2, \qquad c^2=\beta^2\kappa^2.\label{eq008}\end{equation}
\end{lemma}

Let $J_\nu$ denote the Bessel function of the first kind. Then $J_\nu$ solves \eqref{eq007} for all $t\in \R$ and is finite at $0$, which can be taken as the definition of $J_\nu$.

\begin{remark}
$\ $

(i) If we want to solve  \eqref{eq006} with given $a, b, c, \beta$, then we can use \eqref{eq008} to find $\alpha, \kappa, \nu^2$. The general solution of \eqref{eq006} is then
$${\rm span} \left\{t^\alpha J_\nu(\kappa t^\beta),\;
t^\alpha J_{-\nu}(\kappa t^\beta)\right\},$$
assuming $\nu\not\in\Z$.

\ms\no(ii)  Example (the special case that we use in this section):
$$t^2\ddot y+at\dot y+c^2t^2 y=0$$
with parameters $a$ and $c>0$ (and also  $\beta=1$, $b=0$).  We find 
$$\alpha=\frac{1-a}2,\qquad \nu=\alpha,\qquad \kappa=c.$$ General solution:
$${\rm span} \left\{t^\alpha J_\alpha(c t),\;
t^\alpha J_{-\alpha}(c t)\right\}.$$


\end{remark}

\ms\subsection{Bessel canonical system} For each $m>0$ we consider the system with the Hamiltonian
$$H(t)=\begin{pmatrix}
   h(t)   & 0 \\
     0& 1/ h(t)
\end{pmatrix}, \qquad h(t)=t^m,\quad t>0.$$
Note that the system is {\it regular} iff $m<1$.  As usual we introduce the functions $A=A(t,z)$ and $C=C(t,z)$:
$$\Omega\dot X=zHX,\qquad X=(A, C)^\tau,$$
i.e.,
$$\dot C=zhA,\quad -\dot A=\frac z h C,$$
with "initial values"
$$A(0, z)=1, \qquad C(0, z)=0$$
(as limits when $t\to 0$). Rewriting the system as a second order equation for $C$ we get
$$\ddot C-\frac mt\dot C+z^2C=0$$ 
with initial conditions
$$C(0,z)=0,\qquad \dot C(t,z)\sim zt^m\text{ as }t\to 0.$$
For the time being we'll only consider $z\in\R$. According to the example in the last subsection the general solution for the second order 
equation is
$${\rm span} \left\{t^\nu  J_\nu(z t),\;
t^\nu J_{-\nu}(z t)\right\},\qquad \nu=\frac{1+m}2.$$

Consider the function $F_\nu$ defined by
$$J_\nu(\lambda)=\lambda^\nu F_\nu(\lambda).$$
It is known that $F_\nu$ is an entire function and
$$F_\nu(0)=\frac1{2^\nu\Gamma(\nu+1)},\qquad F'_\nu(0)=0.$$
If we fix $z$ and let $t\to 0$, then we have
$$t^\nu J_\nu(zt)=t^{2\nu}z^\nu F_\nu(zt)\sim F_\nu(0) z^\nu t^{2\nu}$$
and 
$$t^\nu J_{-\nu}(zt)=z^{-\nu }F_{-\nu}(zt) \to z^{-\nu}F_{-\nu}(0).$$
From the condition $C(0,z)=0$ it follows that
$$C(t,z)=G(z)t^\nu J_\nu(zt),$$
It remains to find $G(z)$. 

\ms\no We have 
$$C(t,z) \sim G(z)  F_\nu(0) z^\nu t^{2\nu},$$
and
$$ \dot C(t,z) \sim G(z)  F_\nu(0) z^\nu ~ 2\nu t^{2\nu-1}.$$
Combining with the second boundary condition, 
$
 \dot C(t,z)\sim zt^m$, we have
$$zt^m\sim G(z)  F_\nu(0) z^\nu ~ 2\nu t^{2\nu-1},$$
 and therefore (recall that $m=2\nu-1$)
 $$ G(z)=\frac{z^{1-\nu}}{F_\nu(0) ~2\nu}=g_\nu z^{1-\nu},\qquad g_\nu:=2^{\nu-1}\frac{\Gamma(1+\nu)}\nu=2^{\nu-1}\Gamma(\nu).$$

 Thus
 $$C(t,z)=g_\nu t^{2\nu}z F_\nu(zt).$$
 For each fixed $t$, $C$ is entire with respect to $z$. We arrive at
 $$C=g_\nu(t^{2\nu} z F_\nu(zt)),$$
 and 
 $$A=g_\nu(2\nu F_\nu(zt)+tz F'_\nu(zt)).$$
Indeed, we have
 $$\dot C=g_\nu(2\nu t^{2\nu-1} zF_\nu(zt)+t^{2\nu}z^2 F'_\nu(zt)),$$
 and, since $m=2\nu-1$,
 $$A=g_\nu\frac{\dot C}{zt^m}=g_\nu(2\nu F_\nu(zt)+tz F'_\nu(zt)).$$

 \ms\no We can simplify our expression for $A$ to obtain:
 
 \begin{theorem} $$A(t,z)=g_\nu F_{\nu-1}(zt)$$
 	
 	and
 	
 	$$C(t,z)=g_\nu t^{2\nu}z F_\nu(zt).$$
 \end{theorem}
 
 \begin{proof}  As we know
 $$A=H(zt),\qquad H(x):= g_\nu(2\nu F_\nu(x)+xF'_\nu(x)).$$
 We will use the standard relation $$x^\nu F'_\nu=J_{\nu-1}-2\nu x^{-1}J_\nu$$
 (both sides are equal to $-J_{\nu+1}$). We have
 $$x^{\nu-1}H/g_\nu =2\nu x^{\nu-1}F_\nu+x^\nu F'_\nu$$$$=2\nu x^{\nu-1}F_\nu+J_{\nu-1}-2\nu x^{-1}J_\nu
 $$$$=2\nu x^{\nu-1}F_\nu+J_{\nu-1}-2\nu x^{\nu-1}F_\nu =J_{\nu-1},$$
 and $H=g_\nu F_{\nu-1}$.
 \end{proof} 
 
 \ms\no Let $k_t(z)$ be the reproducing kernel of $\BB(E_t)$, $E:=A-iC$, at zero:
 $$k_t(z)=\frac{C(z)A(0)}{\pi z}=\frac{g_\nu^2 F_{\nu-1}(0)}\pi t^{2\nu} F_\nu(zt).$$
 
 Together with Theorem \ref{t10}, this relation implies that the spectral measure $\mu$ of the system is quasi-homogeneous of order $\nu-1=(m-1)/2$. The diagonal form of the Hamiltonian implies that $\mu$ is even. Altogether, 
 $$\mu=\const |x|^{m}.$$
 
 As we can see, $\mu\not\in \PW$.
 
 \ms\no In conclusion, let us verify our computations by 'deriving' the Hamiltonian. We will apply \eqref{eq003A}, even though
 the system does not satisfy the PW-requirement. The correctness of the  answer obtained  suggests broader use of Theorem \ref{t3} and similar formulas.
 
 According to our previous calculations,
 $$\pi k_t(0)=\frac{g_\nu^2 F_{\nu-1}(0)}\pi t^{2\nu} F_\nu(0)= \frac{1}{2\nu\pi} t^{2\nu}.$$ 
 Via \eqref{eq003A},
 $$h_{11}(t)=\pi\dot k_t(0)=t^{2\nu-1}=t^m.$$

\end{document}